\documentclass{amsart}[10pt]
\usepackage{amsmath}
\usepackage{amsmath,amsthm}
\usepackage{amssymb}
\usepackage{amsfonts}
\usepackage{hyperref}
\usepackage{mathtools}

\newtheorem{theorem}{Theorem}[section]

\theoremstyle{definition}
\newtheorem{definition}[theorem]{Definition}
\newtheorem{example}[theorem]{Example}

\newtheorem{proposition}[theorem]{Proposition}
\newtheorem{corollary}[theorem]{Corollary}
\theoremstyle{remark}
\newtheorem{remark}[theorem]{Remark}

\begin{document}
\title[Hyperideals of
Krasner commutative hyperrings ]{$\delta$-$r$-hyperideals and $\phi$-$\delta$-$r$-hyperideals of
commutative Krasner hyperrings }
\author[P. Xu]{Peng Xu}
\address{School of Computer Science of Information Technology,
Qiannan Normal University for Nationalities, 
Duyun, Guizhou, 558000, P.R. China. Orcid: 0000-0001-7028-9987}
\email{gdxupeng@gzhu.edu.cn}
\author[M. Bolat]{Melis Bolat}
\address{Department of Mathematics, Yildiz Technical University, Istanbul, Turkey.
Orcid: 0000-0002-0667-524X}
\email{melisbolat1@gmail.com }
\author[E. Kaya]{Elif Kaya}
\address{Department of Mathematics and Science Education, Istanbul Sabahattin Zaim
University, Istanbul, Turkey. Orcid: 0000-0002-8733-2734}
\email{elif.kaya@izu.edu.tr}
\author[S. Onar]{Serkan Onar}
\address{Department of Mathematical Engineering, Yildiz Technical University, Istanbul,
Turkey. Orcid: 0000-0003-3084-7694}
\email{serkan10ar@gmail.com}
\author[B.A. Ersoy]{Bayram Ali Ersoy}
\address{Department of Mathematics, Yildiz Technical University, Istanbul, Turkey.
Orcid: 0000-0002-8307-9644}
\email{ersoya@yildiz.edu.tr}
\author[K. Hila]{Kostaq Hila}
\address{Department of Mathematical Engineering, Polytechnic University of Tirana,
Tirana, Albania. Orcid: 0000-0001-6425-2619}
\email{kostaq\_hila@yahoo.com}

\begin{abstract}
This paper deals with an important class of multialgebras, called Krasner hyperrings. Our purpose is to define the expansion of $r$-hyperideals and to
extend this concept to $\phi$-$\delta$-$r$-hyperideal in commutative Krasner hyperrings with nonzero identity. $\delta$-$r$-hyperideals of commutative
Krasner hyperrings are studied. Some properties of $\phi$-$\delta$-$r$-hyperideals are investigated and several examples are provided. 

\end{abstract}
\maketitle

\section{Introduction}

Several authors have been studied prime and primary ideals which are quite
important in commutative rings. In 2001, Zhao \cite{X1} introduced the concept
of expansion of ideals of a commutative ring. An expansion of ideals is a
function $\delta$ that assigns to each ideal $\mathcal{N}$ of $\Re$ to another ideal
$\delta(\mathcal{N})$ of the same ring $\Re$ such that for all ideals $\mathcal{N}$ of $\Re,$
$\mathcal{N}\subseteq\delta(\mathcal{N})$ and if $\mathcal{N}\subseteq \mathcal{M}$ where $\mathcal{N}$ and $\mathcal{M}$ are ideals of
$\Re$, then $\delta(\mathcal{N})\subseteq\delta(\mathcal{M}).$ Let $\delta$ be an ideal
expansion. An ideal $\mathcal{N}$ of $\Re$ is called a $\delta$-primary ideal of $\Re,$
if $ab\in \mathcal{N},$ then $a\in \mathcal{N}$ or $b\in\delta(\mathcal{N}),$ for all $a,b\in\Re.$

Further, the concepts of $\phi$-prime, $\phi$-primary ideals are examined in
\cite{X2,X3}.\ A reduction of ideals is a function $\phi$ that maps each ideal
$\mathcal{N}$ of $\Re$ to another ideal $\phi(\mathcal{N})$ of the same ring $\Re$ satisfying the
following terms: $i)$ for all ideals $\mathcal{N}$ of $\Re,$ $\phi(\mathcal{N})\subseteq \mathcal{N}$, $ii)$
if $\mathcal{N}\subseteq \mathcal{M}$ where $\mathcal{N}$ and $\mathcal{M}$ are ideals of $\Re$, then $\phi
(\mathcal{N})\subseteq\phi(\mathcal{M})$. Suppose that $\phi$ is an ideal reduction. A proper
ideal $\mathcal{N}$ of $\Re$ is called a $\phi$-prime ideal of $\Re,$ if $ab\in
\mathcal{N}-\phi(\mathcal{N}),$ then $a\in \mathcal{N}$ or $b\in \mathcal{N},$ for all $a,b\in\Re.$ $\mathcal{N}$ is called
$\phi$-primary ideal of $\Re,$ if $ab\in \mathcal{N}-\phi(\mathcal{N}),$ then $a\in \mathcal{N}$ or
$b\in\sqrt{\mathcal{N}},$ for all $a,b\in\Re.$ Later,\ Jaber \cite{X4} characterized
$\phi$-$\delta$-primary ideals in commutative rings. Assume that $\delta$ is
an ideal expansion and $\phi$ is an ideal reduction. An ideal $\mathcal{N}$ of $\Re$ is
called a $\phi$-$\delta$- primary ideal of $\Re,$ if $ab\in \mathcal{N}-\phi(\mathcal{N}),$ then
$a\in \mathcal{N}$ or $b\in\delta(\mathcal{N}),$ for all $a,b\in\Re$.

Recently, Mohamadian \cite{X5} investigated properties of the class of $r$-ideals. A proper ideal $\mathcal{N}$ in a
commutative ring $\Re$ is called an $r$-ideal (resp., $pr$-ideal), if $ab\in
\mathcal{N}$ with $ann(a)=0$ implies that $b\in \mathcal{N}$ (resp., $b^{n}\in \mathcal{N}$, for some $n\in%
%TCIMACRO{\U{2115} }%
%BeginExpansion
\mathbb{N}
%EndExpansion
$), for each $a,b\in\Re$. Ugurlu \cite{X6} studied generalizations of
$r$-ideals. Assume that $\Re$ is commutative ring, $\mathcal{N}$ is a proper ideal of
$\Re$ and $\phi:Id(\Re)\rightarrow Id(\Re)\cup\{\emptyset\}$ is a function.
$\mathcal{N}$ is called $\phi$-$r$-ideal (resp., $\phi$-$pr$-ideal) if $ab\in \mathcal{N}-\phi(\mathcal{N})$
with $ann(a)=0$ implies that $b\in \mathcal{N}$ (resp., $b^{n}\in \mathcal{N},$ for some $n\in%
%TCIMACRO{\U{2115} }%
%BeginExpansion
\mathbb{N}
%EndExpansion
$)$,$ for $a,b\in\Re.$

Hyperstructures were introduced at the 8th Congress of Scandinavian
Mathematicians in 1934 by a French mathematician Marty \cite{X7}. In the sense
of Marty, let $G\neq\emptyset,$ a mapping $\circ:G\times G\longrightarrow
\mathcal{P}^{\ast}(G)$ is a hyperoperation and $(G,\circ)$ is hypergroupoid. If $A$ and $B$ are two non-empty subsets of $G$ and $x \in G$, then we define: $A \circ B = \bigcup \limits_{\substack{a\in A \atop b\in B}} a\circ b$, $A \circ x = A \circ \lbrace x \rbrace$ and $x \circ B = \lbrace x \rbrace \circ B$. If
$(G,\circ)$ is a hypergroupoid and $x\circ(y\circ z)=(x\circ y)\circ z,$ for
$\forall$ $x,y,z\in G,$ which means that $\bigcup \limits_{\substack{u\in x \circ y}} u \circ z = \bigcup \limits_{\substack{v\in y \circ z}} x \circ v$, then $G$ is a semihypergroup. When $(G,\circ)$ is a
semihypergroup with $x\circ G=G\circ x=G,$ for $\forall x\in G,$ then
$(G,\circ)$ is called a hypergroup. A nonempty set $G$ along with the
hyperoperation $+$ is called a canonical hypergroup if the following axioms
hold: i) $x+(y+z)=(x+y)+z,$ for $x,y,z\in G;$ ii) $x+y=y+x,$ for $x,y\in
G;$ iii) there exists $0\in G$ such that $x+0=\{x\},$ for any $x\in G;$ iv)
for any $x\in G,$ there exists a unique element $x^{\prime}\in G,$ such that
$0\in x+x^{\prime}$ ($x^{\prime}$ is called as the opposite of $x$ and it is denoted by $-x$). v)
$ z\in x+y$ implies that $y\in-x+z$ and $x\in z-y,$ that is $(G,+)$ is
reversible \cite{X8}. A comprehensive review of the theory of hypergroups can be found in \cite{corsini,davvaz,massouros}. 

After that,  many papers and books concerning hyperstructure theory have
appeared in literature \cite{corsini,corsini1,davvaz,davvaz1,davvaz2,vougiouklis}.
Although
there are different types of hyperrings, in this paper, we consider the Krasner hyperrings \cite{X12} which are one of the most important studied class of hyperrings. If $(\Re,+)$ is a canonical hypergroup, $(\Re,\cdot)$ is a
semigroup having $0$ as a bilaterally absorbing element for all elements of
$\Re$ and $\cdot$ is distributive with respect to $+$, then $(G,+,\cdot)$ is known as
Krasner hyperring (see \cite{krasner1}). After Krasner introduced the definition of hyperring, the study of hyperrings has been of great importance and has a multitude of applications to other disciplines, for example see \cite{d1,d2,massouros2}. Students of Krasner, namely Jean Mittas and D. Stratigopoulos, well-known authors, have studied hyperrings and hyperfields. Other names can be also quoted in this respect with nice contributions: P. Corsini, C. Massouros, A. Nakassis, T. Vougiouklis, T.
Koguentsof, A. Dramalidis, S. Spartalis, B. Davvaz, R. Ameri, V. Leoreanu-Fotea, I. Cristea, G. Pinotsis, Y. Kemprasit
and many others. Some fundamental notions and results in the theory of hyperrings can be found in \cite{davvaz3,Massouros1,nakassis,mirela,vougiouklis1}. 

Let $\Re$ be a hyperring. For $a\in\Re,$ we define $ann(a)=\{r\in\Re:ra=0\}$.
	If $ann(a)=0$ (resp., $ann(a)\neq0$)$,$ $a$ is said to be regular (resp., zero
	divisor). We use the notion $r(\Re)$ (resp., $zd(\Re))$ to denote the set of all regular
	elements (resp., zero divisors) \cite{X17}. If $\mathcal{N}$ is a hyperideal of $\Re$ and $A$ a subset of $\Re$, then we denote  $(\mathcal{N}:A)=\{x\in \Re, x\cdot A\subseteq \mathcal{N}\}$. It is clear that $(0 : A) = ann(A)$.
	
According to the definiton of Davvaz, a proper hyperideal $\mathcal{N}$ in a commutative
Krasner hyperring $\Re$ is called an $r$-hyperideal (resp., $pr$-hyperideal),
if $a\cdot b\subseteq \mathcal{N}$ with $ann(a)=0$ implies that $b\in \mathcal{N}$ (resp.,
$b^{n}\in \mathcal{N}$, for some $n\in%
%TCIMACRO{\U{2115} }%
%BeginExpansion
\mathbb{N}
%EndExpansion
$), for each $a,b\in\Re$ \cite{X13}. Some other results on $\phi-\delta$-primary hyperideals and $r$-hyperideals in Krasner hyperrings are obtained recently \cite{kaya,Xu}.

Assume that $\Re$ is a Krasner hyperring, $\phi:Id(\Re)\rightarrow Id(\Re
)\cup\{\emptyset\}$ be a function and $\emptyset\neq \mathcal{N}\in Id(\Re).$ Then $\mathcal{N}$
is said to be a $\phi$-prime (resp. $\phi$-primary) hyperideal of $\Re$ if
whenever $r,s\in\Re$ and $r\cdot s\in \mathcal{N}-$ $\phi(\mathcal{N}),$ then $r\in \mathcal{N}$ or $s\in \mathcal{N}$
(resp., $r\in \mathcal{N}$ or $s^{n}\in \mathcal{N}$). Let $\Re$ be a Krasner hyperring, $\mathcal{N}$ be a
proper ideal of $\Re$ and $\phi:Id(\Re)\rightarrow Id(\Re)\cup\{\emptyset\}$
be a function. $\mathcal{N}$ is called $\phi$-$r$-hyperideal ($\phi$-$pr$-hyperideal) if
$a\cdot b\in \mathcal{N}-\phi(\mathcal{N})$ with $ann(a)=0$ implies that $b\in \mathcal{N}$ ($b^{n}\in \mathcal{N}$),
for $a,b\in\Re.$

In this paper, firstly we study $\delta$-$r$-hyperideals of commutative
Krasner hyperrings. We present the main theorem on $\delta$-$r$-hyperideals. We
show that the intersection of $\delta$-$r$-hyperideals is a $\delta$%
-$r$-hyperideal. Therefore, we show the image and inverse image of a $\delta
$-$r$-hyperideal is also a $\delta$-$r$-hyperideal. In the last section, we
investigate some properties of $\phi$-$\delta$-$r$-hyperideals. We prove that if
$M_{i}$ is a directed set of $\phi$-$\delta$-$r$-hyperideals of $\Re,$ then the
union of $M_{i}$ is a $\phi$-$\delta$-$r$-hyperideal. Also, it is shown that the image and
inverse image of a $\phi$-$\delta$-$r$-hyperideal is $\phi$-$\delta$%
-$r$-hyperideal. As a conclusion, we find out its relation between von Neumann
regular hyperideals \cite{X7} and pure hyperideals.

\section{$\delta$-$r$-hyperideals}

In this paper, $(\Re,+,\cdot)$ denotes commutative Krasner hyperring with
nonzero identity. $r(\Re)$ means the set of all regular elements of $\Re.$

Remember that an expansion of hyperideal, or simply hyperideal expansion, is a
function $\delta$ that assigns each hyperideal $\mathcal{N}$ of a hyperring $\Re$ to
another hyperideal $\delta(\mathcal{N})$ of the same hyperring if the following statements hold:

(i) $\mathcal{N}$ $\subseteq$ $\delta(\mathcal{N})$

(ii) $\mathcal{P}\subseteq Q$ implies $\delta(\mathcal{P})\subseteq\delta(Q)$ for $\mathcal{P},Q$
hyperideals of $\Re$ \cite{X14}.

Moreover, let $\delta$ be a hyperideal exapnsion. A hyperideal $\mathcal{N}$ of
$\Re$ is said to be $\delta$-primary if $a\cdot b\in \mathcal{N}$ and $a\notin \mathcal{N}$ imply
$b\in\delta(\mathcal{N})$, for all $a,b\in\Re$ \cite{X14}.

\begin{definition}
Let $\delta$ be an expansion of hyperideals, $\Re$ be a Krasner hyperring and
$\mathcal{N}$ be a proper hyperideal of $\Re$. $\mathcal{N}$ is called $\delta$-$r$-hyperideal if
$a\cdot b\in \mathcal{N}$ with $ann(a)=0$ implies $b\in\delta(\mathcal{N})$, for all $a,b\in\Re$.
\end{definition}

\begin{example}
(1) If $\mathcal{N}$ is $r$-hyperideal with $\delta_{0}(\mathcal{N})=\mathcal{N}$, then $\mathcal{N}$ is $\delta_{0}$-$r$-hyperideal.

(2) If $\mathcal{N}$ is p$r$-hyperideal with $\delta_{1}(\mathcal{N})=\sqrt{\mathcal{N}},$ then $\mathcal{N}$ is
$\delta_{1}$-$r$-hyperideal.

(3) If $\mathcal{N}$ is $r$-hyperideal with $\delta_{r}(\mathcal{N})=\Re,$ then $\mathcal{N}$ is $\delta_{r}$-$r$-hyperideal.
\end{example}

\begin{remark}
Let $\delta$ be an expansion of hyperideals. It is clear that the following statement hold:

i) $ \mathcal{N}$ is $\phi$-$r$-hyperideal $\Rightarrow \mathcal{N}$ is $r$-hyperideal
$\Rightarrow \mathcal{N}$ is $\delta$-$r$-hyperideal.

ii) If $\mathcal{N}$ is $\delta$-$r$-hyperideal, then $(\mathcal{N}:\mathcal{M})$ is $\delta$%
-$r$-hyperideal, for any subset $\mathcal{M}$ of $\Re$.

iii) If $\delta(\mathcal{N})$ is $\delta$-$r$-hyperideal, then $\mathcal{N}$ is $\delta$-$r$-hyperideal.
\end{remark}

\begin{remark}
i) Assume that $\delta$ and $\gamma$ are hyperideal expansions and
$\delta(\mathcal{N})\subseteq\gamma(\mathcal{N}),$ for each proper hyperideal $\mathcal{N}.$ Then any $\delta
$-$r$-hyperideal is $\gamma$-r-hyperideal.

ii) Let $\delta_{1}$ and $\delta_{2}$ be two hyperideal expansions,
$\mathcal{N}$ be an $r$-hyperideal and $\delta(\mathcal{N})=$ $\delta_{1}(\mathcal{N})\cap$ $\delta_{2}(\mathcal{N}).$
Then, $\delta$ is also a hyperideal expansion.

iii) Let $\delta$ be an expansion of hyperideals. $E_{\delta}%
(\mathcal{P})=\bigcap\{\mathcal{M}\in Id(\Re):\mathcal{P}\subseteq \mathcal{M},$ $\mathcal{M}$ is $\delta$-r-hyperideal$\}.$
Then $E_{\delta}$ is a hyperideal expansion.

For all $\mathcal{P}\in Id(\Re)$, by the definition of $E_{\delta}(\mathcal{P})$, it is clear that $\mathcal{P}\subseteq E_{\delta}(\mathcal{P}).$ We show that for any
$K,L\in Id(\Re)$, if $K\subseteq L$, then $E_{\delta}(K)\subseteq E_{\delta
}(L)$. Indeed: for any $K,L\in Id(\Re),$ if $K\subseteq L,$ then the $\delta
$-r-hyperideal that contains $L$ also contains $K$. Furthermore, there may be
$\delta$-r-hyperideals that contain $L$ but do not contain $K$. As a
result, $E_{\delta}(K)\subseteq E_{\delta}(L).$
\end{remark}

\begin{theorem}
Let $\Re$ satisfy the strong annihilator condition, $\mathcal{N}$ be a proper
hyperideal, $\delta$ be an expansion of hyperideals. $\mathcal{P}$ is $\delta
$-r-hyperideal if and only if for finitely generated hyperideal $\mathcal{N}$ and every
hyperideal $\mathcal{M}$ of $\Re$ such that $\mathcal{N}\cdot \mathcal{M}\subseteq \mathcal{P}$ and $ann(\mathcal{N})=0$ implies
$\mathcal{M}\subseteq\delta(\mathcal{P}).$
\end{theorem}

\begin{proof}
Assume that $\mathcal{P}$ is $\delta$-$r$-hyperideal, $\mathcal{N}\cdot \mathcal{M}\subseteq \mathcal{P},$ $ann(\mathcal{N})=0$
and $\mathcal{M}\nsubseteq\delta(\mathcal{P}).$ Then there exists $a\in \mathcal{N}, $such that
$ann(a)=ann(\mathcal{N})=0$ and $b\in \mathcal{M}-\delta(\mathcal{P}).$ Thus, $a\cdot b\in \mathcal{N}\cdot \mathcal{M}\subseteq
\mathcal{P},$ $ann(a)=0$ and $b\notin\delta(\mathcal{P}).$ This is a contradiction.

Conversely, let $a\cdot b\in \mathcal{P}$ and $ann(a)=0.$ Then $<a>\cdot<b>\subseteq \mathcal{P}$
and $ann(<a>)=0.$ Because of the assumption, $<b>\subseteq\delta(\mathcal{P}).$ Thus,
$b\in\delta(\mathcal{P}).$ Hence, $\mathcal{P}$ is $\delta$-$r$-hyperideal.
\end{proof}

\begin{theorem}
Let $\delta$ be a hyperideal expansion. Then the following statements are equivalent:

i) $ \mathcal{N}$ is $\delta$-r-hyperideal;

ii) $(\mathcal{N}:a)\subseteq\delta(\mathcal{N}),$ for $a\in r(\Re);$

iii) $(\mathcal{N}:\mathcal{P})\subseteq\delta(\mathcal{N}),$ $\mathcal{P}$ is a hyperideal of $\Re$ such that
$\mathcal{P}\cap r(\Re)\nsubseteq \mathcal{N}.$
\end{theorem}

\begin{proof}
(i)$\Rightarrow$(ii) Let $x\in$ $(\mathcal{N}:a),$ for $a\in r(\Re)-\delta(\mathcal{N}).$ Then
$x\cdot a\in \mathcal{N}$ and $ann(a)=0.$ Since $\mathcal{N}$ is $\delta$-$r$-hyperideal,
$x\in\delta(\mathcal{N}).$

(ii)$\Rightarrow$(iii) It is obvious.

(iii)$\Rightarrow$(i) Let $a\cdot x\in \mathcal{N}$ and $\mathcal{P}\cap r(\Re)\nsubseteq \mathcal{N}.$ Then
for an element $x$ of $\Re,$ $x\in \mathcal{P}\cap r(\Re)$ and $x\notin \mathcal{N}.$ This means
$ann(x)=0$ and $a\in(\mathcal{N}:\mathcal{P})\subseteq\delta(\mathcal{N}).$ Thus, $\mathcal{N}$ is $\delta$-$r$-hyperideal.
\end{proof}

\begin{theorem}
Let $\delta$ be a hyperideal expansion that preserves intersection. If
$Q_{1},Q_{2},...,Q_{n}$ are $\delta$-$r$-hyperideals of $\Re$ and $\mathcal{P}=\delta
(Q_{i})$ for all $i$, then $Q=\bigcap_{i=1}^{n}Q_{i}$ is $\delta$-r-hyperideal.
\end{theorem}

\begin{proof}
Let $\delta$ be a hyperideal expansion that preserves intersection, $x\cdot
y\in Q$ and $ann(x)=0.$ Then, $x\cdot y\in Q_{k}$ for some $k.$ Since $Q_{k}$
is $\delta$-$r$-hyperideal, then $y\in\delta(Q_{k}).$ $\delta(Q)=\delta
(\bigcap_{i=1}^{n}Q_{i})=\bigcap_{i=1}^{n}\delta(Q_{i})=\mathcal{P}=\delta(Q_{k}).$
Hence, $y\in\delta(Q).$
\end{proof}

\begin{theorem}
Let $\delta$ be global and $\varphi:\Re\rightarrow S$ be a good epimorphism
and $\mathcal{N}$ be a $\delta$-$r$-hyperideal of $S$. Then $\varphi^{-1}(\mathcal{N})$ is $\delta
$-$r$-hyperideal of $\Re$.
\end{theorem}

\begin{proof}
Let $a\cdot b\in\varphi^{-1}(\mathcal{N})$ and $ann(a)=0,$ for $a,b\in\Re.$ Then,
$\varphi(a\cdot b)=\varphi(a)\cdot\varphi(b)\in \mathcal{N}$ and $ann(\varphi(a))=0.$
Since $\mathcal{N}$ is a $\delta$-$r$-hyperideal, then $\varphi(b)\in\delta(\mathcal{N}).$ Thus,
$b\in\varphi^{-1}(\delta(\mathcal{N}))=\delta(\varphi^{-1}(\mathcal{N})).$ Consequently,
$\varphi^{-1}(\mathcal{N})$ is $\delta$-$r$-hyperideal of $\Re.$
\end{proof}

\begin{theorem}
Let $\varphi:\Re\rightarrow S$ be a good epimorphism and $\mathcal{N}$ be a hyperideal
such that $\ker(\varphi)\subseteq \mathcal{N}$. Then, $\mathcal{N}$ is $\delta$-primary
r-hyperideal of $S$ if and only if $\varphi(\mathcal{N})$ is $\delta$-primary
r-hyperideal of $\Re.$
\end{theorem}

\begin{proof}
If $\varphi(\mathcal{N})$ is $\delta$-$r$-hyperideal, since $\mathcal{N}=\varphi^{-1}(\varphi(\mathcal{N}))$,
by the previous theorem, then $\mathcal{N}$ is $\delta$-$r$-hyperideal.

Let $a,b\in S,$ $a\cdot b\in$ $\varphi(\mathcal{N})$ and $ann(a)=0.$ There are
$x,y\in\Re$ such that $\varphi(x)=a$ and $\varphi(y)=b.$ Since $a\cdot
b=\varphi(x)\cdot$ $\varphi(y)=\varphi(x\cdot y)\in$ $\varphi(\mathcal{N}),$ then $x\cdot
y\in\varphi^{-1}(\varphi(\mathcal{N}))=\mathcal{N}.$ We need to show that $ann(x)=0.$ Let us suppose that
$ann(x)\neq0.$ After that, there is $0\neq c\in\Re$ such that $x\cdot c=0.$
Then, $\varphi(x\cdot c)=\varphi(x)\cdot\varphi(c)=\varphi(0)=0.$ Since
$\varphi(c)\neq0,$ then $ann(\varphi(x))=ann(a)=0.$ This is a contradiction. Thus
$ann(x)=0.$ Since $\mathcal{N}$ is $\delta$-$r$-hyperideal, then $y\in\delta(\mathcal{N}).$ Hence
$b=\varphi(y)\in\varphi(\delta(\mathcal{N}))=\delta(\varphi(\mathcal{N})),$ since $\varphi$ is
onto$.$ Therefore, $\varphi(\mathcal{N})$ is $\delta$-$r$-hyperideal.
\end{proof}

\section{$\phi$-$\delta$-$r$-hyperideals}

We define the reduction of hyperideals to be a function $\phi$ that maps
each hyperideal $\mathcal{N}$ of $\Re$ to another hyperideal $\phi(\mathcal{N})$ of the same
ring $\Re$, satisfying the following axioms:

$i)$ For all hyperideals $\mathcal{N}$ of $\Re,$ $\phi(\mathcal{N})\subseteq \mathcal{N}$.

$ii)$ If $\mathcal{N}\subseteq \mathcal{M}$, where $\mathcal{N}$ and $\mathcal{M}$ are hyperideals of $\Re$, then
$\phi(\mathcal{N})\subseteq\phi(\mathcal{M}).$

\begin{definition}
Let $\delta$ be a hyperideal expansion and $\phi$ a hyperideal reduction. A
hyperideal $\mathcal{N}$ of $\Re$ is called a $\phi$-$\delta$-primary hyperideal if
$a\cdot b\in \mathcal{N}-\phi(\mathcal{N}),$ then $a\in \mathcal{N}$ or $b\in\delta(\mathcal{N}),$ for all $a,b\in
\Re.$
\end{definition}

\begin{definition}
Let $\delta$ be a hyperideal expansion, $\phi$ be a hyperideal reduction. A
proper hyperideal $\mathcal{N}$ of $\Re$ is called a $\phi$-$\delta$-r-hyperideal if
$a\cdot b\in \mathcal{N}-\phi(\mathcal{N})$ with $ann(a)=0$ implies $b\in\delta(\mathcal{N}),$ for all
$a,b\in\Re.$
\end{definition}

Let $\phi$ be a hyperideal reduction and $\delta$ be a hyperideal expansion. Then:

$\mathcal{N}$ is $\phi$-$\delta_{0}$-$r$-hyperideal $\Leftrightarrow \mathcal{N}$ is $\phi$-$r$-hyperideal.

$\mathcal{N}$ is $\phi$-$\delta_{1}$-$r$-hyperideal $\Leftrightarrow$ $\mathcal{N}$ is $\phi$-$pr$-hyperideal.

$\mathcal{N}$ is $\delta$-$r$-hyperideal $\Leftrightarrow \mathcal{N}$ is $\phi_{\emptyset}%
$-$\delta$-$r$-hyperideal.

$\mathcal{N} $ is $r$-hyperideal $\Leftrightarrow \mathcal{\mathcal{N}}$ is $\phi_{\emptyset}$-$\delta_{0}%
$-$r$-hyperideal.

$\mathcal{\mathcal{N}} $ is $pr$-hyperideal $\Leftrightarrow \mathcal{N}$ is $\phi_{\emptyset}$-$\delta_{1}%
$-$r$-hyperideal.

\begin{remark}
(1) If $\delta,\gamma$ are two hyperideal expansions with $\delta
\leq\gamma$ and $\phi$ is any
hyperideal reduction, then every $\phi$-$\delta$-$r$-hyperideal of $\Re$ is $\phi$-$\gamma$-$r$-hyperideal.

(2) If $\delta_{1},...,\delta_{n}$ are hyperideal expansions, then
$\delta=\bigcap\limits_{i=1}^{n}$ $\delta_{i}$ is also a hyperideal expansion.
\end{remark}

\begin{proposition}
Let $\{M_{i}:i\in D\}$ be a directed set of $\phi$-$\delta$-$r$-hyperideals of
$\Re$, where $\phi$ is a hyperideal reduction and $\delta$ is a hyperideal
expansion. Then the hyperideal $\mathcal{M}=\bigcup\limits_{i\in D}M_{i}$ is $\phi
$-$\delta$-$r$-hyperideal.
\end{proposition}

\begin{proof}
Let $a\cdot b\in \mathcal{M}-\phi(\mathcal{M})$ with $ann(a)=0,$ for $a,b\in\Re.$ Since $a\cdot
b\notin\phi(\mathcal{M})$, then $a\cdot b\notin\phi(\mathcal{M}).$ Let $i\in D$ such that $a\cdot b\in
M_{i}-\phi(M_{i})$. Since $ann(a)=0$ and $M_{i}$ is a $\phi$-$\delta
$-r-hyperideal, then $b\in\delta(M_{i})\subseteq\delta(\mathcal{M}).$ Hence, $\mathcal{M}$ is $\phi
$-$\delta$-$r$-hyperideal.
\end{proof}

\begin{theorem}
Let $\mathcal{N}$ be a proper hyperideal of $\Re.$ Suppose that $\phi$ is hyperideal
reduction and $\delta$ is a hyperideal expansion. Then the following statements are equivalent:

i) $ \mathcal{N}$ is $\phi$-$\delta$-$r$-hyperideal;

ii) $(\mathcal{N}:a)\subseteq\delta(\mathcal{N})\cup(\phi(\mathcal{N}):a),$ for $a\in r(\Re);$

iii) $(\mathcal{N}:\mathcal{P})\subseteq\delta(\mathcal{N})\cup(\phi(\mathcal{N}):\mathcal{P}),$ $\mathcal{P}$ is a hyperideal of
$\Re$ such that $\mathcal{P}\cap r(\Re)\nsubseteq \mathcal{N}.$
\end{theorem}

\begin{proof}
(i)$\Rightarrow$(ii) Let $x\in(\mathcal{N}:a).$ Then $x\cdot a\in \mathcal{N}.$ If $x\cdot
a\in\phi(\mathcal{N}),$ then $x\in(\phi(\mathcal{N}):a).$ If $x\cdot a\notin\phi(\mathcal{N}),$ then
$x\in\delta(\mathcal{N}).$

(ii)$\Rightarrow$(iii) It is obvious.

(iii)$\Rightarrow$(i) Let $a\cdot x\in \mathcal{N}-\phi(\mathcal{N})$ and $\mathcal{P}\cap r(\Re)\nsubseteq
\mathcal{N}.$ Then for an element $x$ of $\Re,$ we have $x\in \mathcal{P}\cap r(\Re)$ and $x\notin \mathcal{N}.$
This means $ann(x)=0$ and $a\in(\mathcal{N}:\mathcal{P})-(\phi(\mathcal{N}):\mathcal{P}).$ Since $a\notin(\phi(\mathcal{N}):\mathcal{P}),$ then
$a\in$ $\delta(\mathcal{N}).$
\end{proof}

\begin{theorem}
Let $\Re$ satisfy the strong annihilator condition and $\mathcal{N}$ be a proper
hyperideal. $\mathcal{P}$ is $\phi$-$\delta$-$r$-hyperideal if and only if $\mathcal{N}\cdot
\mathcal{M}\subseteq \mathcal{P}-\phi(\mathcal{N})$ and $ann(\mathcal{N})=0$ implies $\mathcal{M}\subseteq\delta(\mathcal{P})$ where $\mathcal{N}$
is finitely generated hyperideal and $\mathcal{M}$ is a hyperideal of $\Re.$
\end{theorem}

\begin{proof}
Assume that $\mathcal{P}$ is $\phi$-$\delta$-$r$-hyperideal, $\mathcal{N}\cdot \mathcal{M}\subseteq
\mathcal{P}-\phi(\mathcal{N}),$ $ann(\mathcal{N})=0$ and $\mathcal{M}\nsubseteq\delta(\mathcal{P}).$ Then there exists $a\in
\mathcal{N},$ such that $ann(a)=ann(\mathcal{N})=0$ and $b\in \mathcal{M}-\delta(\mathcal{P}).$ Thus, $a\cdot b\in
\mathcal{N}\cdot \mathcal{M}\subseteq \mathcal{P}-\phi(\mathcal{N}),$ $ann(a)=0$ and $b\notin\delta(\mathcal{P}).$ This is a contradiction.

Conversely, let $a\cdot b\in \mathcal{P}-\phi(\mathcal{N})$ and $ann(a)=0.$ Then $<a>\cdot
<b>\subseteq \mathcal{P}-\phi(\mathcal{N})$ and $ann(<a>)=0.$ Because of the assumption,
$<b>\subseteq\delta(\mathcal{P}).$ Thus, $b\in\delta(\mathcal{P}).$ Hence, $\mathcal{P}$ is $\phi$-$\delta
$-$r$-hyperideal.
\end{proof}

\begin{theorem}
Let $\phi$ be hyperideal reduction and $\delta$ be a hyperideal expansion. If
$\mathcal{P}$ is a $\phi$-$\delta$-$r$-hyperideal of $\Re$ and $\phi(\mathcal{P}:a)=(\phi(\mathcal{P}):a),$
for $a\in\Re,$ then $(\mathcal{P}:a)$ is also $\phi$-$\delta$-$r$-hyperideal of $\Re.$
\end{theorem}

\begin{proof}
Let $a\cdot b\in(\mathcal{P}:a)-\phi(\mathcal{P}:a)$ with $ann(x)=0.$ Then $x\cdot y\cdot a\in
\mathcal{P}-\phi(\mathcal{P})$ and $ann(x\cdot a)=0.$ Since $\mathcal{P}$ is $\phi$-$\delta$-$r$-hyperideal and
$\phi(\mathcal{P}:a)=(\phi(\mathcal{P}):a)$, then $b\in\delta(\mathcal{P})\subseteq\delta(\mathcal{P}:a).$ Hence, $(\mathcal{P}:a)$
is $\phi$-$\delta$-$r$-hyperideal of $\Re.$
\end{proof}

A reduction is said to be global if for any hyperring homomorphism
$\varphi:\Re\rightarrow S,$ $\phi(\varphi^{-1}(\mathcal{N}))=\varphi^{-1}(\phi(\mathcal{N})),$
for all $\mathcal{N}\in Id(S).$ For example, the reductions $\phi_{1}$ and $\phi_{2}$
are both global.

\begin{theorem}
\label{the9} Let $\phi$ be a hyperideal reduction and $\delta$ be a
hyperideal expansion, all of which are global. If $\varphi:\Re\rightarrow S$
is a good epimorphism, then for any $\phi$-$\delta$-$r$-hyperideal $\mathcal{N}$ of $S$,
$\varphi^{-1}(\mathcal{N})$ is a $\phi$-$\delta$-$r$-hyperideal of $\Re$.
\end{theorem}

\begin{proof}
Suppose that $\mathcal{N}$ is a $\phi$-$\delta$-$r$-hyperideal of $S$. Let $a\cdot b\in$
$\varphi^{-1}(\mathcal{N})-\phi(\varphi^{-1}(\mathcal{N}))$ with $ann(a)=0.$ Then $\varphi(a\cdot
b)=\varphi(a)\cdot\varphi(b)\in \mathcal{N}-\varphi(\mathcal{N}).$ If $ann(\varphi(a))\neq0,$
there is a $0\neq\varphi(x)\in S$, for $x\in\Re,$ such that $\varphi
(a)\cdot\varphi(x)=\varphi(a\cdot x)=0.$ This means there is a $0\neq
x\in\varphi^{-1}(\mathcal{N})$ such that $a\cdot x=0,$ that is a contradiction. Hence,
$ann(\varphi(a))=0.$ Since $\mathcal{N}$ is a $\phi$-$\delta$-$r$-hyperideal, then
$\varphi(b)\in\delta(\mathcal{N}).$ And so, $b\in\varphi^{-1}(\delta(\mathcal{N})).$
\end{proof}

\begin{theorem}
Let $\phi$ be a hyperideal reduction and $\delta$ be a hyperideal
expansion such that both of them are global. Let $\varphi:\Re\rightarrow S$ be
a good epimorphism. Then any $\phi$-$\delta$-$r$-hyperideal $\mathcal{M}$ of $\Re$
contains $\ker\varphi$ if and only if $\varphi(\mathcal{M})$ is a $\phi$-$\delta
$-$r$-hyperideal of $S$.
\end{theorem}

\begin{proof}
If $\varphi(\mathcal{M})$ is a $\phi$-$\delta$-$r$-hyperideal of $S$, then by previous
theorem, $\mathcal{M}=$ $\varphi^{-1}(\varphi(\mathcal{M}))$ is a $\phi$-$\delta$-$r$-hyperideal of
$\Re.$ Conversely, let $b_{1}\cdot$ $b_{2}\in\varphi(\mathcal{M})-\phi(\varphi(\mathcal{M}))$ and
$ann(b_{1})=0$ for $b_{1},$ $b_{2}\in S.$ Since $\varphi$ is onto, then there exist
$a_{1},$ $a_{2}\in\Re$ such that $b_{1}=\varphi(a_{1})$ and $b_{2}%
=\varphi(a_{2})$. Then $b_{1}\cdot$ $b_{2}=\varphi(a_{1})\cdot\varphi
(a_{2})=\varphi(a_{1}\cdot a_{2})=\varphi(x)\in\varphi(\mathcal{M})-\phi(\varphi(\mathcal{M}))$,
for some $x\in \mathcal{M}$. $0\in\varphi(a_{1}\cdot a_{2})-\varphi(x)=\varphi
(a_{1}\cdot a_{2}-x).$ Then, there is a $t\in a_{1}\cdot a_{2}-x$ such that
$\varphi(t)=0$: We have $a_{1}\cdot a_{2}\in t+x\subseteq\ker\varphi
+\mathcal{M}\subseteq \mathcal{M}+\mathcal{M}\subseteq \mathcal{M}$: Thus, $a_{1}\cdot a_{2}\in \mathcal{M}-\phi(\mathcal{M}).$ Let
$ann(a_{1})\neq0$: Then there exists $c (\neq 0)\in\Re$ such that $a_{1}\cdot
c=0.$ So, $\varphi(a_{1}\cdot c)=\varphi(a_{1})\cdot\varphi(c)=\varphi(0)=0.$
Since $\varphi(c)\neq0,$ then $ann(\varphi(a_{1}))=ann(b_{1})=0$. This is a
contradiction. So, this implies that $ann(a_{1})=0$: Since $\mathcal{M}$ is $\phi
$-$\delta$-$r$-hyperideal, then $a_{2}\in\delta(\mathcal{M}).$ We have $b_{2}=\varphi(a_{2}%
)\in$ $\varphi(\delta(\mathcal{M})).$ Since $\varphi$ is onto, then $\delta(\mathcal{M})=\delta
(\varphi^{-1}(\varphi(\mathcal{M})))=\varphi^{-1}(\delta(\varphi(\mathcal{M})))$, therefore
$\varphi(\delta(\mathcal{M}))=\delta(\varphi(\mathcal{M})).$ Hence, $b_{2}\in\delta(\varphi(\mathcal{M})).$
\end{proof}

\begin{remark}
From Theorem \ref{the9}, for a global hyperideal expansion $\delta$, if
$\varphi:\Re\rightarrow S$ is a good epimorphism and $\ker\varphi\subseteq \mathcal{N}$,
for some hyperideal $\mathcal{N}$ of $\Re$, then $\varphi(\delta(\mathcal{M}))=\delta(\varphi
(\mathcal{M}))$. Therefore, if $\varphi$ is an isomorphism, then $\varphi(\delta
(\mathcal{M}))=\delta(\varphi(\mathcal{M}))$ for all $\mathcal{N}\in Id(\Re).$
\end{remark}

\begin{corollary}
Let $\phi$ be a hyperideal reduction $\phi$ and $\delta$ be a hyperideal expansion, all of which are global. Let $\mathcal{N},$ $\mathcal{M}$ be any hyperideals of $\Re$
such that $\mathcal{M}$ contains $\mathcal{N}$. Then $\mathcal{M}/\mathcal{N}$ is a $\phi$-$\delta$-$r$-hyperideal of
$\Re/\mathcal{N}$ if and only if $\mathcal{M}$ is a $\phi$-$\delta$-$r$-hyperideal of $\Re$.
\end{corollary}

\begin{proof}
Let $\varphi:\Re\rightarrow\Re/\mathcal{N}$ be a good epimorphism with $\ker
\varphi=\mathcal{N}\subseteq \mathcal{M}.$ By the Theorem 10, $\mathcal{M}/\mathcal{N}$ is a $\phi$-$\delta$%
-$r$-hyperideal of $\Re/\mathcal{N}$ if and only if $\mathcal{M}$ is a $\phi$-$\delta$%
-$r$-hyperideal of $\Re$.
\end{proof}

\begin{definition}
\label{def4}Let $\phi$ is a hyperideal reduction and $\delta$ be a
hyperideal expansion. Let $\mathcal{N}$ be a proper hyperideal.

(1) If for every $a\in \mathcal{N}-\phi(\mathcal{N})$, there is $b\in\delta(\mathcal{N})$ such that
$a=a\cdot b,$ then $\mathcal{N}$ is called $\phi$-$\delta$-pure hyperideal.

(2) If for every $a\in \mathcal{N}-\phi(\mathcal{N})$, there is $b\in\delta(\mathcal{N})$ such that
$a=a^{2}\cdot b,$ then $\mathcal{N}$ is called $\phi$-$\delta$-von Neumann regular hyperideal.
\end{definition}

\begin{corollary}
i) Every $\phi$-$\delta$-pure hyperideal is a $\phi$-$\delta$-$r$-hyperideal.

ii) Every $\phi$-$\delta$-von Neumann regular hyperideal is a $\phi
$-$\delta$-$r$-hyperideal.
\end{corollary}

\begin{proof}
We can simply prove that any $\phi$-$\delta$-pure hyperideal is a $\phi$%
-$\delta$-$r$-hyperideal using the Definition \ref{def4}. Moreover, each
$\phi$-$\delta$-von Neumann regular hyperideal is likewise a $\phi$-$\delta
$-$r$-hyperideal.
\end{proof}

\section{Conclusion}

In this study, we defined the expansion of $r$-hyperideals and extended this
concept to $\phi$-$\delta$-$r$-hyperideal over commutative Krasner hyperring
with nonzero identity. We investigated some of their essential characteristics
and provided some examples. We also looked into the relations between these
structures. The overall framework of these structures is then explained providing
a number of major conclusions. We obtained some specific results explaining
the structures. For instance, we indicated that when $\Re$ satisfies the
strong annihilator condition, $\mathcal{N}$ is a proper hyperideal and $\delta$ is an
expansion of hyperideals, $\mathcal{P}$ is a $\delta$-$r$-hyperideal if and only if
$\mathcal{N}\cdot \mathcal{M}\subseteq \mathcal{P}$ and $ann(\mathcal{N})=0$ implies $\mathcal{M}\subseteq\delta(\mathcal{P})$ for
finitely generatered hyperideal $\mathcal{N}$ and hyperideal $\mathcal{M}$ of $\Re$. Then we also
showed that a similar result is satisfied for $\phi$-$\delta$-$r$-hyperideals of
$\Re$.

This study offers a significant advance to the classification of hyperideals in Krasner
hyperrings. Following this paper, it is feasible to investigate additional
algebraic hyperstructures, such as Krasner $(m,n)$-hyperrings, and do research on their properties. Based on our study,
we present some open problems for future work to researchers as follows: 

1) To describe 2-absorbing $\delta$-$r$-hyperideals;

2) To consider $\delta$-$r$-subhypermodules;

3) To think $\phi$-$\delta$-$r$-subhypermodules.\\\newline

\section{ Acknowledgement}

A preliminary version of this manuscript was submitted as a pre-print in arxiv.org \cite{Xu1}.\\

{\textbf{Funding}}\\
This work was funded in part by the National Natural Science Foundation of China (grant no. 62002079). \\

{\textbf{Compliance with Ethical Standards}}\\

\textbf{Conflict of Interest} All authors declare that they
have no conflict of interest.\\

\textbf{Ethical approval} This article does not contain any
studies with human participants or animals performed
by any of the authors.\\

\textbf{Data Availability Statement}

Data sharing not applicable to this article as no data-sets were generated or analyzed during the current study.

\end{document}